\documentclass{article}

\usepackage{amsmath,amssymb,amsfonts,amsthm,mdframed,graphicx}
\usepackage{color, colortbl}
\usepackage{xcolor}
\usepackage{array}
\usepackage{multirow}

\newtheorem{theorem}{Theorem}
\newtheorem{lemma}[theorem]{Lemma}

\newtheorem{remark}[theorem]{Remark}

\usepackage{tikz}
\usetikzlibrary{arrows.meta}
\tikzset{>={Latex[width=0.8mm,length=0.8mm]}}
\usetikzlibrary{positioning}
\usetikzlibrary{arrows}
\usepackage{tikz,amsmath, amssymb,bm,color}

\title{Human-verifiable proofs in the theory of word-representable graphs}

\author{Sergey Kitaev\footnote{Department of Mathematics and Statistics, University of Strathclyde, 26 Richmond Street, Glasgow G1, 1XH, United Kingdom. 
{\bf Email:} sergey.kitaev@strath.ac.uk.}\ \ and Haoran Sun\footnote{College of Science, Donghua University, Shanghai, 201620, P.R. China. {\bf Email:} sunhaoran2012@outlook.com}}

\begin{document}

\maketitle

\begin{abstract}
A graph is word-representable if it can be represented in a certain way using alternation of letters in words. Word-representable graphs generalise several important and well-studied classes of graphs, and they can be characterised by semi-transitive orientations. Recognising word-representability is an NP-complete problem, and the bottleneck of the theory of word-representable graphs is convincing someone that a graph is non-word-representable, keeping in mind that references to (even publicly available and user-friendly) software are not always welcome. (Word-representability can be justified by providing a semi-transitive orientation as a certificate that can be checked in polynomial time.) 

In the literature, a variety of (usually ad hoc) proofs of non-word-representability for particular graphs, or families of graphs, appear, but for a randomly selected graph, one should expect looking at $O(2^{\#\mbox{edges}})$ orientations and justifying that none of them is semi-transitive. Even if computer would print out all these orientations and would point out what is wrong with each of the orientations, such a proof would be essentially non-checkable by a human. 

In this paper, we develop methods for an automatic search of human-verifiable proofs of graph non-word-representability. As a proof-of-concept, we provide ``short'' proofs of non-word-representability, generated automatically by our publicly available user-friendly software, of the Shrikhande graph on 16 vertices and 48 edges (6 ``lines'' of proof) and the Clebsch graph on 16 vertices and 40 edges (10 ``lines'' of proof). Producing such short proofs for relatively large graphs would not be possible without the instrumental tool we introduce (allowing to assume orientations of several edges in a graph, not just one edge as it was previously used) that is a game changer in the area.  As a bi-product of our studies, we correct two mistakes published multiple times (two graphs out of the 25 non-word-representable graphs on 7 vertices were actually word-representable, while two non-word-representable graphs on 7 vertices were missing).  \\


\end{abstract}

\section{Introduction}

There is a long line of research papers in the literature dedicated to the theory of word-representable graphs (e.g.\ see \cite{K17,KL15} and references therein). The motivation to study these graphs is their relevance to algebra, graph theory, computer science, combinatorics on words, and scheduling \cite{KL15}. In particular, word-representable graphs generalize several fundamental classes of graphs (e.g.\ {\em circle graphs}, {\em $3$-colorable graphs} and {\em comparability graphs}). 

Two letters $x$ and $y$ alternate in a word $w$ if after deleting in $w$ all letters but the copies of $x$ and $y$ we either obtain a word $xyxy\cdots$ (of even or odd length) or a word $yxyx\cdots$ (of even or odd length).  A graph $G=(V,E)$ is {\em word-representable} if and only if there exists a word $w$
over the alphabet $V$ such that letters $x$ and $y$, $x\neq y$, alternate in $w$ if and only if $xy\in E$. We say that $w$ {\em represents} $G$. A graph is word-representable if and only if each of its connected components is word-representable \cite{KL15}. Hence, WLOG we can assume that graphs under our consideration are always connected.

The minimum (by the number of vertices) non-word-representable graph is on 6 vertices, and the only such graph is the wheel graph $W_5$ (that is obtained from Graph 1 in Figure~\ref{25-non-word-repr} after removing vertex 7), while there are 25 non-word-representable graphs on 7 vertices \cite{KL15}. A word is {\em uniform} if each letter occurs in it the same number of times. We will need the following two results.

\begin{theorem}[\cite{KP08}]\label{aux-1} If a graph $G$ is word-representable then there exists a uniform word representing it. \end{theorem}

\begin{theorem}[\cite{KP08}]\label{aux-2} Let $w=w_1w_2$ be a uniform word representing a graph $G$, where $w_1$ and $w_2$ are words. Then the word $w_2w_1$ also represents $G$. \end{theorem}

An orientation of a graph is {\em semi-transitive} if it is acyclic, and for any directed path $v_0\rightarrow v_1\rightarrow \cdots \rightarrow v_k$ either there is no edge between $v_0$ and $v_k$, or $v_i\rightarrow v_j$ is an edge for all $0\leq i<j\leq k$. 
An induced subgraph on vertices $\{v_0,v_1,\ldots,v_k\}$ of an oriented graph is a {\em shortcut} if its orientation is acyclic (contains no directed cycles) and non-transitive, and there is the directed path $v_0\rightarrow v_1\rightarrow \cdots \rightarrow v_k$  and the edge $v_0\rightarrow v_k$ called the {\em shortcutting edge}. A semi-transitive orientation can then be alternatively defined as an acyclic shortcut-free orientation. A fundamental result in the area is the following theorem.

\begin{theorem}[\cite{HKP16}]\label{fundamental} A graph is word-representable if and only if it admits a semi-transitive orientation. \end{theorem} 

\begin{remark}\label{rem-1} An efficient way to obtain a semi-transitive orientation of a graph $G$ from a word $w$ representing it, which is used in the proof of Theorem~\ref{fundamental}, is as follows: orient an edge $xy$ in $G$ as $x\rightarrow y$ if the leftmost occurrence of $x$ in $w$ is to the left of the leftmost occurrence of $y$ in $w$. Note that  the leftmost letter in $w$ will be a {\em source} in $G$, that is a vertex having edges incident with it oriented away from it.\end{remark}

Recognizing word-representability of a graph is an NP-complete problem \cite{KL15}, and the bottleneck of the theory of word-representable graphs is in ways to convince someone that a graph is non-word-representable keeping in mind that references to (even publicly available and user-friendly) software, such as \cite{G}, are not always welcome. (Word-representability can be justified by providing a semi-transitive orientation as a certificate that can be checked in polynomial time~\cite{KL15}.) 

\subsection{Approaches to deal with non-word-representability} It is known \cite{KP08} that the neighbourhood of each node in a word-representable graph is a comparability graph, and recognition of a comparability graph is a polynomially solvable problem \cite{G76}. Hence, we have a polynomial test for non-word-representability of a graph $G$: for each vertex, check whether its neighbourhood is a comparability graph; if a ``non-comparability neighbourhood'' is found, $G$ is not word-representable. However, if all neighbourhoods are comparability graphs, then the test gives no information, namely, the graph may or may not be word-representable \cite{KL15}. 

Thus, if the test above does not work and a known non-word-representable induced subgraph is not found, basically we are left with three options to recognise and then justify non-word-representability: either 
\begin{itemize}
\item[(a)] to come up with some sort of an ad hoc smart argument, usually using properties and/or symmetries of the graph in question, or its induced subgraph, or 
\item[(b)] to go through $O(2^{\#\mbox{edges}})$ orientations justifying that none of them is semi-transitive (symmetries can be used here sometimes to reduce the search space, in particular, any given edge can be assumed to be oriented in any way), or
\item[(c)] to go through all $O(\#\mbox{vertices}^2)$ words containing each of the vertex labels equal number of times and to justify that none of them has the right alternation properties (if a graph with $n$ vertices is word-representable then there is a word of length at most $n^2$ representing it \cite{HKP16}).
\end{itemize} 

Approach (a) above is preferable, but usually is hard to implement. Approach (c) requires going through $O(n^{2n})$ words, however, constraint programming can be used here to speed up the process \cite{Z}. In either case, how do we convince someone that the graph is non-word-representable without a reference to software? A variation of approach (b) is used in some existing pieces of software \cite{G,S}. It works as follows: orient an edge $e_1$ in a given graph $G$, then consider a still undirected edge $e_2$ in $G$ and branch on it, namely, create two copies of the partially oriented graph by orienting $e_2$ differently; then branch on $e_3$, etc. At each step, make sure that no directed cycles or shortcuts are created (if they are, the respective branch is not to be considered). In any case, even if computer would print out all these orientations (or the entire branching process) and would point at a directed cycle or a shortcut in each of the orientations, such a proof would be essentially non-checkable by a human, as it would typically be a way too long.

\subsection{A game changer approach} In this paper we consider producing ``short'' proofs of non-word-representability dropping the number of cases to be considered from exponential to polynomial, and thus enabling human to verify such proofs. The basic idea is in modifying the branching process by avoiding unnecessary branching via certain pre-processing. The following lemma is the key to our approach. 

\begin{lemma}[\cite{KP}]\label{lemma} Suppose that an undirected graph $G$ has a cycle $C=x_1x_2\cdots x_mx_1$, where $m\geq 4$ and the vertices in $\{x_1,x_2,\ldots,x_m\}$ do not induce a clique in $G$.  If $G$ is oriented semi-transitively, and $m-2$ edges of $C$ are oriented in the same direction (i.e. from $x_i$ to $x_{i+1}$ or vice versa, where the index $m+1:=1$) then the remaining two edges of $C$ are oriented in the opposite direction.\end{lemma}

Hence, if we try to find a semi-transitive orientation by exhaustively going through all possibilities to orient one edge at a time, and we see a cycle, that does not induce a clique, with all but two edges oriented in the same direction, we do not need to branch on the remaining two edges as they must be oriented in the opposite direction by Lemma~\ref{lemma}. Similarly, if we see a non-clique cycle with all but two edges oriented in the same direction and an edge $e$ in the cycle oriented in the opposite direction, then we known that the remaining edge is oriented in the same direction as $e$.

The following theorem allows us to reduce further dramatically the length of a proof of non-word-representability for a graph. 

 \begin{theorem}\label{source-thm} Suppose that a graph $G$ is word-representable, and $v$ is a vertex in $G$. Then, there exists a semi-transitive orientation of $G$ where $v$ is a source (or a sink so that all edges incident with $v$ are oriented towards it).   
 \end{theorem}
 
 \begin{proof}
 Let $w$ be a word representing $G$. By Theorem~\ref{aux-1}, we can assume that $w$ is uniform, and by Theorem~\ref{aux-2} that $w$ begins with $v$. Then we can use Remark~\ref{rem-1} to obtain a semi-transitive orientation of $G$ in which $v$ is a source. Reversing the orientations of all edges in $G$, one obtains a semi-transitive orientation in which $v$ is a sink.
 \end{proof}

To demonstrate the power of Theorem~\ref{source-thm}, just applying Lemma~\ref{lemma} in the branching process for the wheel graph $W_5$, one needs to branch 6 times. However, assuming WLOG by Theorem~\ref{source-thm} that the all-adjacent vertex in $W_5$ is a source, one needs to branch just once. In fact, using the symmetry of $W_5$, this only branching is not necessary, so non-word-representability of $W_5$ can be proved by Lemma~\ref{lemma} and Theorem~\ref{source-thm} and observing the symmetry without branching. Usually, picking a vertex of the highest degree in a given non-word-representable graph and assuming it to be a source would result in a shortest proof. However, there are cases when assuming a vertex of a smaller degree to be a source results in a shorter proof. Also, note that assuming a particular vertex to be a source may result in a longer proof than making no assumption at all.

\begin{center}
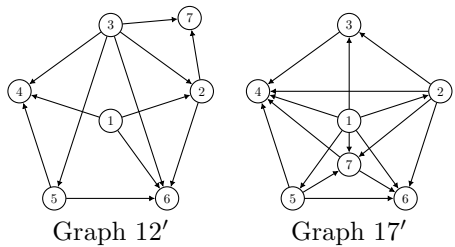
\begin{figure}
\begin{center}
\begin{tabular}{c c c c c}

\begin{tikzpicture}[scale=0.3]

\draw (0,0) node [scale=0.5, circle, draw](node5){5};
\draw (5,0) node [scale=0.5, circle, draw](node6){6};
\draw (6.54,4.75) node [scale=0.5, circle, draw](node2){2};
\draw (2.50,7.69) node [scale=0.5, circle, draw](node3){3};
\draw (-1.54,4.75) node [scale=0.5, circle, draw](node4){4};
\draw (6,8) node [scale=0.5, circle, draw](node7){7};
\draw(2.5,3.44)node [scale=0.5, circle, draw](node1){1};

\draw [->](node1)--(node2);   

\draw [->](node1)--(node4);
 
\draw [->](node1)--(node6);
\draw [->](node2)--(node7);
\draw [->](node2)--(node6);

\draw [->](node3)--(node7);
\draw [->](node3)--(node2);
\draw [->](node3)--(node5);
\draw [->](node3)--(node6);
\draw [->](node3)--(node4);

\draw [->](node5)--(node6);
\draw [->](node5)--(node4);

\draw (2.5,-1.5) node {Graph $12'$};

\end{tikzpicture}

& 

\begin{tikzpicture}[scale=0.3]

\draw (0,0) node [scale=0.5, circle, draw](node5){5};
\draw (5,0) node [scale=0.5, circle, draw](node6){6};
\draw (6.54,4.75) node [scale=0.5, circle, draw](node2){2};
\draw (2.50,7.69) node [scale=0.5, circle, draw](node3){3};
\draw (-1.54,4.75) node [scale=0.5, circle, draw](node4){4};
\draw (2.5,1.5) node [scale=0.5, circle, draw](node7){7};
\draw(2.5,3.44)node [scale=0.5, circle, draw](node1){1};

\draw [->](node1)--(node2);
\draw [->](node1)--(node3);
\draw [->](node1)--(node4);
\draw [->](node1)--(node5);
\draw [->](node1)--(node6);
\draw [->](node1)--(node7);

\draw [->](node2)--(node3);
\draw [->](node2)--(node4);
\draw [->](node2)--(node6);
\draw [->](node2)--(node7);

\draw [->](node3)--(node4);

\draw [->](node5)--(node6);
\draw [->](node5)--(node7);
\draw [->](node5)--(node4);

\draw [->](node7)--(node6);
\draw [->](node7)--(node4);

\draw (2.5,-1.5) node {Graph $17'$};

\end{tikzpicture}

\end{tabular}
\caption{The undirected versions of Graphs $12'$ and $17'$ were assumed to be non-word-representable several times in the literature, for example in \cite{KL15}, although they are actually word-representable as is witnessed by the semi-transitive orientations given in the figure. Graph $12'$ misses the edge (1,3). Graph $17'$ should not have the edge (1,7).}\label{two-graphs}
\end{center}
\end{figure}
\end{center}

In Section~\ref{algorithms}, we introduce three algorithms getting use of Lemma~\ref{lemma} to generate shorter proofs for non-word-representable graphs. The primarily criteria of the efficiency of an algorithm is the number of ``lines'' (in the sense specified below) in the proof it produces; the secondary criteria is the running time of the algorithm, whose discussion is omitted from this paper as less relevant (the problem is NP-complete, so the running time is exponential, and further details are not so important for us). We test and rank our algorithms on the 25 non-word-representable graphs on 7 vertices correcting, as a by-product, two mistakes published multiple times, e.g.\ in \cite{KL15}. Indeed, two graphs out of the 25 graphs were produced incorrectly. These incorrect graphs are the undirected versions of the semi-transitively oriented graphs in  Figure~\ref{two-graphs}. We leave it to the Reader as a straightforward exercise to prove that the orientations in Figure~\ref{two-graphs} are indeed semi-transitive. A correct list of the 25 non-word-representable graphs can be found in Figure~\ref{25-non-word-repr}. We note that in our test we choose not to use Theorem~\ref{source-thm}; using this theorem and going through all possible assumptions about a source may give a different outcome for comparison of the algorithms.

We would like to emphasise that the approach involving Lemma~\ref{lemma} is not novel: several papers, including \cite{KP}, use it or its simpler version (considering cycles of length 3 and 4 only) to justify non-word-representability. However, our paper is the first one to discuss an automated search of human-verifiable proofs of graph non-word-representability that allowed us to create publicly available user-friendly software \cite{S}. Moreover, our Theorem~\ref{source-thm} enlarges significantly the class of graphs for which ``reasonably short'' proofs of non-word-representability can be produced. As a proof-of-concept,  we use the software \cite{S} to find ``short'' proofs of non-word-representability, generated automatically, of the {\em Shrikhande graph} on 16 vertices and 48 edges (6 ``lines'' of proof; see Section~\ref{Shrikhande}) and the {\em Clebsch graph} on 16 vertices and 40 edges (10 ``lines'' of proof; see Section~\ref{Clebsch}). Proving in a verifiable way (without referring to computer software) that the Clebsch graph and the Shrikhande graph are non-word-representable would be extremely challenging without the tools we introduce.  

\section{Three algorithms to search for short proofs of non-word-representability}\label{algorithms}

In this section, we consider three algorithms to find shorter proofs of non-word-representability of graphs. All three algorithms use the observation that the branching process should not involve any edges that do not belong to a cycle, as such edges can be oriented arbitrarily (they will never be involved in a directed cycle of a shortcut). Further, all three algorithms use the assumption that to produce a shorter proof, branching should be made on edges belonging to many cycles (which should increase the number of applications of Lemma~\ref{lemma}). 

\subsection{The format of a proof}

By a ``line'' of a proof we mean a sequence of instructions that directs us in orienting a partially oriented graph and necessarily ends with detecting a shortcut or another contradiction showing that this particular orientation branch will not produce a semi-transitive orientation. The idea is that if no branch produces a semi-transitive orientation then the graph is non-semi-transitively orientable (and hence non-word-representable by Theorem~\ref{fundamental}). 

Each proof begins with $A\rightarrow B$ showing the orientation of an edge $AB$, the first edge we orient, that is selected by an algorithm in a certain way. Because reversing all orientations in a semi-transitively oriented graph produces a semi-transitively oriented graph, WLOG we can omit considering (partially) oriented graphs having $B\rightarrow A$, which significantly reduces the number of cases to consider. Further, there are four types of instructions:
\begin{itemize}
\item ``MC'' followed by a number $X$ means ``Move to Copy $X$'', where Copy $X$ of the graph in question is a partially oriented version of the graph that was created at some point in the branching process. This instruction is always followed by an oriented edge $A\rightarrow B$ reminding on the directed edge obtained after application of the branching process; see description of ``B'' to be discussed next. 
\item ``B'' followed by ``$X\rightarrow Y$ (Copy $Z$)'' means ``Branch on edge $XY$, orient the edge as $X\rightarrow Y$, create a copy of the current version of the graph except orient the edge $XY$ there as $Y\rightarrow X$, and call the new copy $Z$; leave Copy $Z$ aside and continue to follow the instructions''. The instruction B occurs when the software detects that no application of Lemma~\ref{lemma} is possible in the partially oriented graph.
\item One ``O'' followed by ``$X\rightarrow Y$'', in turn followed by, in brackets, ``C'' followed by a cycle ``$X$-$Y$-$Z$''. This instruction tells us to orient the edge $XY$ as $X\rightarrow Y$ because otherwise, in the triangle $XYZ$, we would get a directed cycle. If instead of a triangle we see a longer cycle, then we deal with an application of Lemma~\ref{lemma} to a cycle where all but two edges are oriented in one direction, and one of the remaining two edges is oriented in the opposite direction.
\item Two ``O''s followed by ``$X\rightarrow Y$'', in turn followed by, in brackets, ``C'' followed by a cycle ``$X$-$Y$-$Z$-$\cdots$''. This instruction tells us to which cycle Lemma~\ref{lemma} can be applied and which edges will become oriented.
\end{itemize}

Each line ends with either ``$S:X-Y-\cdots-Z$'' or with ``$E:X-Y-\cdots-Z$''. In the former case, a shortcut with the shortcutting edge $X\rightarrow Z$ is obtained, while in the latter case, all but one edges in the non-clique cycle $X-Y-\cdots-Z$ are oriented in the same direction, while the remaining edge $e$ is not oriented, which is a contradiction since there is no way to orient $e$ without creating a shortcut or a directed cycle (``$E$'' stands for ``Error''). In the proofs in this paper, the symbol ``$E$'' does not appear, but it appears occasionally when the software \cite{S} is used. 

Next, we will explain parts of the first two lines, given in Figure~\ref{2-lines}, in the proof of non-word-representability of the Shrikhande graph in Figure~\ref{Shrikhande-pic} that does not used Theorem~\ref{source-thm} (and hence is different from the proof in Section~\ref{Shrikhande}).

\begin{figure}
\noindent\rule{12cm}{0.4pt}

\begin{small}
{\bf 1.} 12$\rightarrow$15 B14$\rightarrow$15 (Copy 2) B12$\rightarrow$14 (Copy 3) O7$\rightarrow$15 O12$\rightarrow$7 (C7-15-14-12) {\bf [other instructions]} S:7-4-8-16

{\bf 2.} MC4 16$\rightarrow$7 O16$\rightarrow$15 (C7-16-15)  {\bf [other instructions]}  S:4-11-3-7
\end{small}

\vspace{-0.2cm}
\noindent\rule{12cm}{0.4pt}
\caption{Parts of the first two lines in the proof of non-word-representability of the Shrikhande graph in Figure~\ref{Shrikhande-pic}}\label{2-lines}
\end{figure}

To begin checking the proof, one should arrange 9 undirected copies of the Shrikhande graph, possibly printed on a single page. Begin with orienting edge 12$\rightarrow$15 in the first copy of the graph. Branching is necessary at this stage, we orient edge 14$\rightarrow$15 in Copy 1 and create partially oriented Copy 2 currently having edges 12$\rightarrow$15  and 15$\rightarrow$14. We continue with considering Copy 1. Another branching is required, and we orient the edge 12$\rightarrow$14 and create partially oriented Copy 3 currently having edges 12$\rightarrow$15, 14$\rightarrow$15 and 14$\rightarrow$12. Looking at the cycle 7-15-14-12 in Copy 1, we can see that Lemma~\ref{lemma} can be applied and we can orient edges 7$\rightarrow$15  and 12$\rightarrow$7. Continuing following the instructions, we see that the shortcut 7-4-8-16 will eventually be obtained in Copy 1 showing that Copy 1 can now be disregarded as any way to complete its orientation will result in a shortcut being present (so that the orientation would be non-semi-transitive). 

We can now consider any of the three partially oriented copies of the graph (Copies 2, 3, 4). Our algorithm suggests considering the latest created copy (Copy 4) that has the most number of oriented edges.  MC4 instructs us to do so, and 16$\rightarrow$7 reminds us on the correct orientation of the edge (7,16) obtained as the result of the branching process (when Copy 4 was created). Next, we look at the triangle 7-16-15 where we must orient edge 16$\rightarrow$15 or else we obtain a directed cycle of length 3. Continuing following the instructions, we see that the shortcut 4-11-3-7 will eventually be obtained in Copy 4 showing that Copy 4 can now be disregarded, and another copy should be considered.

\subsection{Algorithm 1}
Algorithm 1 sorts edges according to the number of cycles they are in, then branches on an edge belonging to the most number of cycles (whenever branching is required). If there are two or more such edges, the choice on branching is done lexicographically.

\subsection{Algorithm 2}
Algorithm 2 selects a cycle $C$ with the smallest number of non-oriented edges. The non-oriented edges in $C$ are sorted, similarly to Algorithm 1, based on  the number of cycles they are in and branching is done on an edge belonging to the most number of cycles. If there are two or more such edges, the choice on branching is done lexicographically.

\subsection{Algorithm 3}
Algorithm 3 is similar to Algorithm 2, but it selects a cycle hat has the biggest number $N$ of edges oriented in the same direction. Among the cycles with the same $N$, Algorithm 3 selects a cycle $C$ that has smallest number of non-oriented edges. Then, similarly to Algorithm 2, the non-oriented edges in $C$ are sorted based on  the number of cycles they are in and branching is done on an edge belonging to the most number of cycles. If there are two or more such edges, the choice on branching is done lexicographically.

\subsection{Ranking of algorithms}

\begin{figure}
\begin{center}
\begin{tabular}{c c c c c}
\begin{tikzpicture}[scale=0.25]

\draw (0,0) node [scale=0.5, circle, draw](node5){5};
\draw (5,0) node [scale=0.5, circle, draw](node6){6};
\draw (6.54,4.75) node [scale=0.5, circle, draw](node2){2};
\draw (2.50,7.69) node [scale=0.5, circle, draw](node3){3};
\draw (-1.54,4.75) node [scale=0.5, circle, draw](node4){4};
\draw (5,8.5) node [scale=0.5, circle, draw](node7){7};
\draw(2.5,3.44)node [scale=0.5, circle, draw](node1){1};

\draw (node6)--(node2)--(node3)--(node4)--(node5)--(node6)  ;
\draw (node1)--(node2);
\draw (node1)--(node3);
\draw (node1)--(node4);
\draw (node1)--(node5);
\draw (node1)--(node6);
\draw (node3)--(node7);
\draw (2.5,-1.5) node {Graph 1};
\end{tikzpicture} 

&
 
\begin{tikzpicture}[scale=0.25]

\draw (0,0) node [scale=0.5, circle, draw](node5){5};
\draw (5,0) node [scale=0.5, circle, draw](node6){6};
\draw (6.54,4.75) node [scale=0.5, circle, draw](node2){2};
\draw (2.50,7.69) node [scale=0.5, circle, draw](node3){3};
\draw (-1.54,4.75) node [scale=0.5, circle, draw](node4){4};
\draw (2.5,9.5) node [scale=0.5, circle, draw](node7){7};
\draw(2.5,3.44)node [scale=0.5, circle, draw](node1){1};

\draw (node6)--(node2)--(node3)--(node4)--(node5)--(node6)  ;
\draw (node1)--(node2);
\draw (node1)--(node3);
\draw (node1)--(node4);
\draw (node1)--(node5);
\draw (node1)--(node6);
\draw (node4)--(node7);
\draw (node2)--(node7);
\draw (2.5,-1.5) node {Graph 2};
\end{tikzpicture}
& 
\begin{tikzpicture}[scale=0.25]

\draw (0,0) node [scale=0.5, circle, draw](node5){5};
\draw (5,0) node [scale=0.5, circle, draw](node6){6};
\draw (6.54,4.75) node [scale=0.5, circle, draw](node2){2};
\draw (2.50,7.69) node [scale=0.5, circle, draw](node3){3};
\draw (-1.54,4.75) node [scale=0.5, circle, draw](node4){4};
\draw (2.5,9.5) node [scale=0.5, circle, draw](node7){7};
\draw(2.5,3.44)node [scale=0.5, circle, draw](node1){1};

\draw (node6)--(node2)--(node3)--(node4)--(node5)--(node6)  ;
\draw (node1)--(node2);
\draw (node1)--(node3);
\draw (node1)--(node4);
\draw (node1)--(node5);
\draw (node1)--(node6);
\draw (node4)--(node7);
\draw (node2)--(node7);
\draw (node3)--(node7);
\draw (2.5,-1.5) node {Graph 3};
\end{tikzpicture} 					
& 
\begin{tikzpicture}[scale=0.25]

\draw (0,0) node [scale=0.5, circle, draw](node5){5};
\draw (5,0) node [scale=0.5, circle, draw](node6){6};
\draw (6.54,4.75) node [scale=0.5, circle, draw](node2){2};
\draw (2.50,7.69) node [scale=0.5, circle, draw](node3){3};
\draw (-1.54,4.75) node [scale=0.5, circle, draw](node4){4};
\draw (2.5,9.5) node [scale=0.5, circle, draw](node7){7};
\draw(2.5,3.44)node [scale=0.5, circle, draw](node1){1};

\draw (node6)--(node2)--(node3)--(node4)--(node5)--(node6)  ;
\draw (node1)--(node2);
\draw (node1)--(node3);
\draw (node1)--(node4);
\draw (node1)--(node5);
\draw (node1)--(node6);
\draw (node4)--(node7);
\draw (node2)--(node7);
\draw    (node1) to[out=60,in=-60] (node7);
\draw (2.5,-1.5) node {Graph 4};

\end{tikzpicture}

&
 
\begin{tikzpicture}[scale=0.25]

\draw (0,0) node [scale=0.5, circle, draw](node5){5};
\draw (5,0) node [scale=0.5, circle, draw](node6){6};
\draw (6.54,4.75) node [scale=0.5, circle, draw](node2){2};
\draw (2.50,7.69) node [scale=0.5, circle, draw](node3){3};
\draw (-1.54,4.75) node [scale=0.5, circle, draw](node4){4};
\draw (6,8) node [scale=0.5, circle, draw](node7){7};
\draw(2.5,3.44)node [scale=0.5, circle, draw](node1){1};

\draw (node6)--(node2)--(node3)--(node4)--(node5)--(node6)  ;
\draw (node1)--(node2);
\draw (node1)--(node3);
\draw (node1)--(node4);
\draw (node1)--(node5);
\draw (node1)--(node6);
\draw (node1)--(node7);
\draw (2.5,-1.5) node {Graph 5};

\end{tikzpicture}
\\

\begin{tikzpicture}[scale=0.25]

\draw (0,0) node [scale=0.5, circle, draw](node5){5};
\draw (5,0) node [scale=0.5, circle, draw](node6){6};
\draw (6.54,4.75) node [scale=0.5, circle, draw](node2){2};
\draw (2.50,7.69) node [scale=0.5, circle, draw](node3){3};
\draw (-1.54,4.75) node [scale=0.5, circle, draw](node4){4};
\draw (6,8) node [scale=0.5, circle, draw](node7){7};
\draw(2.5,3.44)node [scale=0.5, circle, draw](node1){1};

\draw (node6)--(node2)--(node3)--(node4)--(node5)--(node6)  ;
\draw (node1)--(node2);
\draw (node1)--(node3);
\draw (node1)--(node4);
\draw (node1)--(node5);
\draw (node1)--(node6);
\draw (node1)--(node7);
\draw (node3)--(node7);

\draw (2.5,-1.5) node {Graph 6};

\end{tikzpicture}
& 
\begin{tikzpicture}[scale=0.25]

\draw (0,0) node [scale=0.5, circle, draw](node5){5};
\draw (5,0) node [scale=0.5, circle, draw](node6){6};
\draw (6.54,4.75) node [scale=0.5, circle, draw](node2){2};
\draw (2.50,7.69) node [scale=0.5, circle, draw](node3){3};
\draw (-1.54,4.75) node [scale=0.5, circle, draw](node4){4};
\draw (6,8) node [scale=0.5, circle, draw](node7){7};
\draw(2.5,3.44)node [scale=0.5, circle, draw](node1){1};

\draw (node6)--(node2)--(node3)--(node4)--(node5)--(node6)  ;
\draw (node1)--(node2);
\draw (node1)--(node3);
\draw (node1)--(node4);
\draw (node1)--(node5);
\draw (node1)--(node6);
\draw (node2)--(node7);
\draw (node3)--(node7);

\draw (2.5,-1.5) node {Graph 7};

\end{tikzpicture}
&  
\begin{tikzpicture}[scale=0.25]

\draw (0,0) node [scale=0.5, circle, draw](node5){5};
\draw (5,0) node [scale=0.5, circle, draw](node6){6};
\draw (6.54,4.75) node [scale=0.5, circle, draw](node2){2};
\draw (2.50,7.69) node [scale=0.5, circle, draw](node3){3};
\draw (-1.54,4.75) node [scale=0.5, circle, draw](node4){4};
\draw (6,8) node [scale=0.5, circle, draw](node7){7};
\draw(2.5,3.44)node [scale=0.5, circle, draw](node1){1};

\draw (node6)--(node2)--(node3)--(node4)--(node5)--(node6)  ;
\draw (node1)--(node2);
\draw (node1)--(node3);
\draw (node1)--(node4);
\draw (node1)--(node5);
\draw (node1)--(node6);
\draw (node2)--(node7);
\draw (node3)--(node7);
\draw (node1)--(node7);
\draw (2.5,-1.5) node {Graph 8};

\end{tikzpicture}
&

\begin{tikzpicture}[scale=0.25]

\draw (0,0) node [scale=0.5, circle, draw](node5){5};
\draw (5,0) node [scale=0.5, circle, draw](node6){6};
\draw (6.54,4.75) node [scale=0.5, circle, draw](node2){2};
\draw (2.50,7.69) node [scale=0.5, circle, draw](node3){3};
\draw (-1.54,4.75) node [scale=0.5, circle, draw](node4){4};
\draw (6,8) node [scale=0.5, circle, draw](node7){7};
\draw(2.5,3.44)node [scale=0.5, circle, draw](node1){1};

\draw (node6)--(node2)--(node3)--(node4)--(node5)--(node6)  ;
\draw (node1)--(node2);
\draw (node1)--(node3);
\draw (node1)--(node4);
\draw (node1)--(node5);
\draw (node1)--(node6);
\draw (node2)--(node7);
\draw (node3)--(node7);
\draw (node5)to[out=20,in=-100](node7);
\draw (2.5,-1.5) node {Graph 9};

\end{tikzpicture}
&
\begin{tikzpicture}[scale=0.25]

\draw (0,0) node [scale=0.5, circle, draw](node5){5};
\draw (5,0) node [scale=0.5, circle, draw](node6){6};
\draw (6.54,4.75) node [scale=0.5, circle, draw](node2){2};
\draw (2.50,7.69) node [scale=0.5, circle, draw](node3){3};
\draw (-1.54,4.75) node [scale=0.5, circle, draw](node4){4};
\draw (6,8) node [scale=0.5, circle, draw](node7){7};
\draw(2.5,3.44)node [scale=0.5, circle, draw](node1){1};

\draw (node6)--(node2)--(node3)--(node4)--(node5)--(node6)  ;
\draw (node1)--(node2);
\draw (node1)--(node3);
\draw (node1)--(node4);
\draw (node1)--(node5);
\draw (node1)--(node6);
\draw (node2)--(node7);
\draw (node3)--(node7);
\draw (node5)to[out=20,in=-100](node7);
\draw (node1)--(node7);
\draw (2.5,-1.5) node {Graph 10};

\end{tikzpicture}
\\

\begin{tikzpicture}[scale=0.25]

\draw (0,0) node [scale=0.5, circle, draw](node5){5};
\draw (5,0) node [scale=0.5, circle, draw](node6){6};
\draw (6.54,4.75) node [scale=0.5, circle, draw](node2){2};
\draw (2.50,7.69) node [scale=0.5, circle, draw](node3){3};
\draw (-1.54,4.75) node [scale=0.5, circle, draw](node4){4};
\draw (6,8) node [scale=0.5, circle, draw](node7){7};
\draw(2.5,3.44)node [scale=0.5, circle, draw](node1){1};

\draw (node6)--(node2)--(node3)--(node4)--(node5)--(node6)  ;
\draw (node1)--(node2);
\draw (node1)--(node3);
\draw (node1)--(node4);
\draw (node1)--(node5);
\draw (node1)--(node6);

\draw (node2)--(node7);
\draw (node3)--(node7);
\draw (node6)--(node7);
\draw (node1)--(node7);
\draw (2.5,-1.5) node {Graph 11};

\end{tikzpicture}
& 
\begin{tikzpicture}[scale=0.25]

\draw (0,0) node [scale=0.5, circle, draw](node5){5};
\draw (5,0) node [scale=0.5, circle, draw](node6){6};
\draw (6.54,4.75) node [scale=0.5, circle, draw](node2){2};
\draw (2.50,7.69) node [scale=0.5, circle, draw](node3){3};
\draw (-1.54,4.75) node [scale=0.5, circle, draw](node4){4};
\draw (6,8) node [scale=0.5, circle, draw](node7){7};
\draw(2.5,3.44)node [scale=0.5, circle, draw](node1){1};

\draw (node6)--(node2)--(node3)--(node4)--(node5)--(node6)  ;
\draw (node1)--(node2);
\draw (node1)--(node3);
\draw (node1)--(node4);
 
\draw (node1)--(node6);
\draw (node2)--(node7);
\draw (node3)--(node7);
\draw (node3)--(node5);
\draw (node3)--(node6);
\draw (node3)--(node1);
\draw (2.5,-1.5) node {Graph 12};

\end{tikzpicture}

&

\begin{tikzpicture}[scale=0.25]

\draw (0,0) node [scale=0.5, circle, draw](node5){5};
\draw (5,0) node [scale=0.5, circle, draw](node6){6};
\draw (6.54,4.75) node [scale=0.5, circle, draw](node2){2};
\draw (2.50,7.69) node [scale=0.5, circle, draw](node3){3};
\draw (-1.54,4.75) node [scale=0.5, circle, draw](node4){4};
\draw (2.5,1.5) node [scale=0.5, circle, draw](node7){7};
\draw(2.5,3.44)node [scale=0.5, circle, draw](node1){1};

\draw (node6)--(node2)--(node3)--(node4)--(node5)--(node6)  ;

\draw (node1)--(node3);

\draw (node1)--(node5);
\draw (node1)--(node6);
\draw (node7)--(node6);
\draw (node7)--(node2);
\draw (node7)--(node4);
\draw (node7)--(node5);
\draw (2.5,-1.5) node {Graph 13};

\end{tikzpicture}
& 
\begin{tikzpicture}[scale=0.25]

\draw (0,0) node [scale=0.5, circle, draw](node5){5};
\draw (5,0) node [scale=0.5, circle, draw](node6){6};
\draw (6.54,4.75) node [scale=0.5, circle, draw](node2){2};
\draw (2.50,7.69) node [scale=0.5, circle, draw](node3){3};
\draw (-1.54,4.75) node [scale=0.5, circle, draw](node4){4};
\draw (2.5,1.5) node [scale=0.5, circle, draw](node7){7};
\draw(2.5,3.44)node [scale=0.5, circle, draw](node1){1};

\draw (node6)--(node2)--(node3)--(node4)--(node5)--(node6)  ;

\draw (node1)--(node3);

\draw (node1)--(node5);
\draw (node1)--(node6);
\draw (node7)--(node6);
\draw (node7)--(node2);
\draw (node7)--(node4);
\draw (node7)--(node5);
\draw (node1)--(node4);
\draw (2.5,-1.5) node {Graph 14};

\end{tikzpicture}
& 
\begin{tikzpicture}[scale=0.25]

\draw (0,0) node [scale=0.5, circle, draw](node5){5};
\draw (5,0) node [scale=0.5, circle, draw](node6){6};
\draw (6.54,4.75) node [scale=0.5, circle, draw](node2){2};
\draw (2.50,7.69) node [scale=0.5, circle, draw](node3){3};
\draw (-1.54,4.75) node [scale=0.5, circle, draw](node4){4};
\draw (2.5,1.5) node [scale=0.5, circle, draw](node7){7};
\draw(2.5,3.44)node [scale=0.5, circle, draw](node1){1};

\draw (node6)--(node2)--(node3)--(node4)--(node5)--(node6)  ;

\draw (node1)--(node3);

\draw (node1)--(node5);
\draw (node1)--(node6);
\draw (node7)--(node6);
\draw (node7)--(node2);
\draw (node7)--(node4);
\draw (node7)--(node5);

\draw (node1)--(node4);
\draw (node2)--(node1);

\draw (2.5,-1.5) node {Graph 15};

\end{tikzpicture}
\\

\begin{tikzpicture}[scale=0.25]

\draw (0,0) node [scale=0.5, circle, draw](node5){5};
\draw (5,0) node [scale=0.5, circle, draw](node6){6};
\draw (6.54,4.75) node [scale=0.5, circle, draw](node2){2};
\draw (2.50,7.69) node [scale=0.5, circle, draw](node3){3};
\draw (-1.54,4.75) node [scale=0.5, circle, draw](node4){4};
\draw (2.5,1.5) node [scale=0.5, circle, draw](node7){7};
\draw(2.5,3.44)node [scale=0.5, circle, draw](node1){1};

\draw (node6)--(node2)--(node3)--(node4)--(node5)--(node6)  ;

\draw (node1)--(node3);

\draw (node1)--(node5);
\draw (node1)--(node6);
\draw (node7)--(node6);
\draw (node7)--(node2);
\draw (node7)--(node4);
\draw (node7)--(node5);

\draw (node1)--(node4);
\draw (node2)--(node1);
\draw (node7)--(node1);
\draw (2.5,-1.5) node {Graph 16};

\end{tikzpicture}
&

\begin{tikzpicture}[scale=0.25]

\draw (0,0) node [scale=0.5, circle, draw](node5){5};
\draw (5,0) node [scale=0.5, circle, draw](node6){6};
\draw (6.54,4.75) node [scale=0.5, circle, draw](node2){2};
\draw (2.50,7.69) node [scale=0.5, circle, draw](node3){3};
\draw (-1.54,4.75) node [scale=0.5, circle, draw](node4){4};
\draw (2.5,1.5) node [scale=0.5, circle, draw](node7){7};
\draw(2.5,3.44)node [scale=0.5, circle, draw](node1){1};

\draw (node6)--(node2)--(node3)--(node4)--(node5)--(node6)  ;

\draw (node1)--(node3);

\draw (node1)--(node5);
\draw (node1)--(node6);
\draw (node7)--(node6);
\draw (node7)--(node2);
\draw (node7)--(node4);
\draw (node7)--(node5);

\draw (node1)--(node4);
\draw (node2)--(node1);
\draw (node2)--(node4);
\draw (2.5,-1.5) node {Graph 17};

\end{tikzpicture}
&
\begin{tikzpicture}[scale=0.25]

\draw (0,0) node [scale=0.5, circle, draw](node5){5};
\draw (5,0) node [scale=0.5, circle, draw](node6){6};
\draw (6.54,4.75) node [scale=0.5, circle, draw](node2){2};
\draw (2.50,7.69) node [scale=0.5, circle, draw](node3){3};
\draw (-1.54,4.75) node [scale=0.5, circle, draw](node4){4};
\draw (2.5,1.5) node [scale=0.5, circle, draw](node7){7};
\draw(2.5,3.44)node [scale=0.5, circle, draw](node1){1};

\draw (node6)--(node2)--(node3)--(node4)--(node5)--(node6)  ;

\draw (node1)--(node3);

\draw (node1)--(node5);
\draw (node1)--(node6);
\draw (node7)--(node6);
\draw (node7)--(node2);
\draw (node7)--(node4);
\draw (node7)--(node5);

\draw (node1)--(node4);
\draw (node2)--(node1);
\draw (node7)to[out=60,in=-60](node3);
\draw (2.5,-1.5) node {Graph 18};

\end{tikzpicture}
&
\begin{tikzpicture}[scale=0.25]

\draw (0,0) node [scale=0.5, circle, draw](node5){5};
\draw (5,0) node [scale=0.5, circle, draw](node6){6};
\draw (6.54,4.75) node [scale=0.5, circle, draw](node2){2};
\draw (2.50,7.69) node [scale=0.5, circle, draw](node3){3};
\draw (-1.54,4.75) node [scale=0.5, circle, draw](node4){4};
\draw (2.5,1.5) node [scale=0.5, circle, draw](node7){7};
\draw(2.5,3.44)node [scale=0.5, circle, draw](node1){1};

\draw (node6)--(node2)--(node3)--(node4)--(node5)--(node6)  ;

\draw (node1)--(node3);

\draw (node1)--(node5);
\draw (node1)--(node6);
\draw (node7)--(node6);
\draw (node7)--(node2);
\draw (node7)--(node4);
\draw (node7)--(node5);
\draw (node7)--(node1);
\draw (node1)--(node4);
\draw (node2)--(node1);
\draw (node7)to[out=60,in=-60](node3);
\draw (2.5,-1.5) node {Graph 19};

\end{tikzpicture}
&
\begin{tikzpicture}[scale=0.25]

\draw (0,0) node [scale=0.5, circle, draw](node5){5};
\draw (5,0) node [scale=0.5, circle, draw](node6){6};
\draw (6.54,4.75) node [scale=0.5, circle, draw](node2){2};
\draw (2.50,7.69) node [scale=0.5, circle, draw](node3){3};
\draw (-1.54,4.75) node [scale=0.5, circle, draw](node4){4};
\draw (2.5,1.5) node [scale=0.5, circle, draw](node7){7};
\draw(2.5,3.44)node [scale=0.5, circle, draw](node1){1};

\draw (node6)--(node2)--(node3)--(node4)--(node5)--(node6)  ;

\draw (node1)--(node3);

\draw (node1)--(node5);
\draw (node1)--(node6);
\draw (node7)--(node6);
\draw (node7)--(node2);
\draw (node7)--(node4);
\draw (node7)--(node5);
\draw (node7)--(node1);
\draw (node1)--(node4);
\draw (node2)--(node1);
\draw (node2)--(node4);

\draw (node7)to[out=60,in=-60](node3);
\draw (2.5,-1.5) node {Graph 20};

\end{tikzpicture}

\\

\begin{tikzpicture}[scale=0.25]

\draw (0,0) node [scale=0.5, circle, draw](node1){1};
\draw (7.54,0) node [scale=0.5, circle, draw](node2){2};
\draw (0,7.54) node [scale=0.5, circle, draw](node4){4};
\draw (7.54,7.54) node [scale=0.5, circle, draw](node3){3};
\draw (3,3) node [scale=0.5, circle, draw](node6){6};
\draw (3,1.5) node [scale=0.5, circle, draw](node5){5};
\draw(1.5,3)node [scale=0.5, circle, draw](node7){7};

\draw (node1)--(node2)--(node3)--(node4)--(node1)  ;
\draw (node4)--(node2);
\draw (node3)--(node6);
\draw (node1)--(node6);
\draw (node2)--(node6);
\draw (node7)--(node6);
\draw (node1)--(node7);
\draw (node1)--(node5);
\draw (node6)--(node5);
\draw (node2)--(node5);

\draw (4,-1.5) node {Graph 21};
\end{tikzpicture}

&

\begin{tikzpicture}[scale=0.25]

\draw (0,0) node [scale=0.5, circle, draw](node6){6};
\draw (4,6.84) node [scale=0.5, circle, draw](node5){5};
\draw (8,0) node [scale=0.5, circle, draw](node7){7};
\draw (4,2.33) node [scale=0.5, circle, draw](node1){1};
\draw (2,3.42) node [scale=0.5, circle, draw](node3){3};
\draw (6, 3.42) node [scale=0.5, circle, draw](node2){2};
\draw (4,0)node [scale=0.5, circle, draw](node4){4};

\draw (node6)--(node3)--(node5)--(node2)--(node7)--(node4)--(node6)  ;
\draw (node4)--(node2)--(node3)--(node4);
\draw (node1)--(node2);
\draw (node1)--(node3);
\draw (node1)--(node4);

\draw (4,-1.5) node {Graph 22};
\end{tikzpicture}

&

\begin{tikzpicture}[scale=0.25]

\draw (0,0) node [scale=0.5, circle, draw](node6){6};
\draw (4,6.84) node [scale=0.5, circle, draw](node5){5};
\draw (8,0) node [scale=0.5, circle, draw](node7){7};
\draw (4,2.33) node [scale=0.5, circle, draw](node1){1};
\draw (2,3.42) node [scale=0.5, circle, draw](node3){3};
\draw (6, 3.42) node [scale=0.5, circle, draw](node2){2};
\draw (4,0)node [scale=0.5, circle, draw](node4){4};

\draw (node6)--(node3)--(node5)--(node2)--(node7)--(node4)--(node6)  ;
\draw (node4)--(node2)--(node3)--(node4);
\draw (node1)--(node2);
\draw (node1)--(node3);
\draw (node1)--(node4);

\draw (node1)--(node5);
\draw (node1)--(node6);
\draw (node1)--(node7);
\draw (4,-1.5) node {Graph 23};
\end{tikzpicture} 
&
\begin{tikzpicture}[scale=0.25]

\draw (0,0) node [scale=0.5, circle, draw](node6){6};
\draw (4,6.84) node [scale=0.5, circle, draw](node5){5};
\draw (8,0) node [scale=0.5, circle, draw](node7){7};
\draw (4,2.33) node [scale=0.5, circle, draw](node1){1};
\draw (2,3.42) node [scale=0.5, circle, draw](node3){3};
\draw (6, 3.42) node [scale=0.5, circle, draw](node2){2};
\draw (4,0)node [scale=0.5, circle, draw](node4){4};

\draw (node6)--(node3)--(node5)--(node2)--(node7)--(node4)--(node6)  ;
\draw (node4)--(node2)--(node3)--(node4);

\draw (node1)--(node3);

\draw (node1)--(node5);
\draw (node1)--(node6);
\draw (node1)--(node7);
\draw (4,-1.5) node {Graph 24};
\end{tikzpicture} 
&
\begin{tikzpicture}[scale=0.3]

\draw (0,0) node [scale=0.5, circle, draw](node6){6};
\draw (4,6.84) node [scale=0.5, circle, draw](node5){5};
\draw (8,0) node [scale=0.5, circle, draw](node7){7};
\draw (4,2.33) node [scale=0.5, circle, draw](node1){1};

\draw (2.8,2.84) node [scale=0.5, circle, draw](node3){3};
\draw (5.2, 2.84) node [scale=0.5, circle, draw](node2){2};
\draw (4,1)node [scale=0.5, circle, draw](node4){4};

\draw (node6)--(node5)--(node7)--(node6)  ;

\draw (node1)--(node2);
\draw (node1)--(node3);
\draw (node1)--(node4);

\draw (node1)--(node5);
\draw (node1)--(node6);
\draw (node1)--(node7);
\draw (node2)--(node7);
\draw (node3)--(node5);
\draw (node4)--(node6);
\draw (4,-1.5) node {Graph 25};
\end{tikzpicture} 

 \\
\end{tabular}
\caption{All non-word-representable graphs on 7 vertices}\label{25-non-word-repr}
\end{center}
\end{figure}
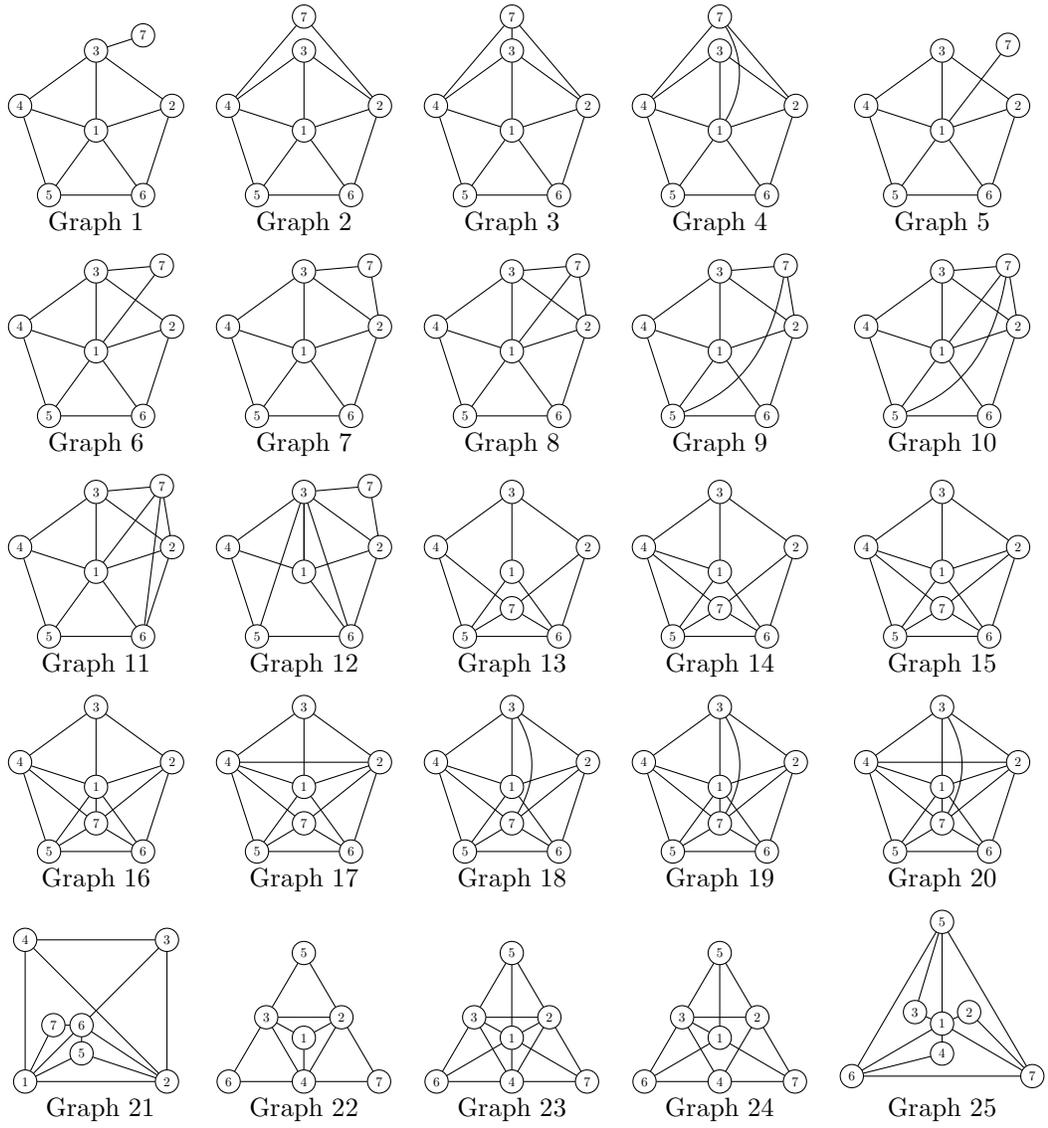

Note that Algorithm 1 is static while Algorithms 2 and 3 are dynamic meaning that they require resorting edges whenever an orientation is added to an edge. 

To make general statements on the efficiency of algorithms in the sense of the number of lines they produce, or about the time complexity, does not seem to be feasible. However, an indication of the efficiency of the algorithms can be obtained by looking at their performance on small non-word-representable graphs. For example, on the wheel graph $W_5$ (on 6 vertices), Algorithm 1 produces 10 lines of proof, while Algorithms 2 and 3 produce 7 lines of proof. As the next step, we test the algorithms on all 25 non-word-representable graphs in Figure~\ref{25-non-word-repr}, and the results of the test are presented in Table~\ref{ranking-table}. Recall that in our tests we do not use Theorem~\ref{source-thm}. It turns out that Algorithm 2 is (much) better/not worse in 24 out of 25 cases, and what is somewhat surprising, Algorithm 1 being clearly the worst one, has actually the best performance on Graph 11. On average, Algorithms 2 and 3 are essentially the same. In any case, Algorithm 2 is used in the software \cite{S}.

\begin{table}[htb]
\centering
\begin{tabular}{c|c|c|c|}
\hline
\rowcolor{gray!20!}
{\bf Graph} & {\bf Algorithm 2} &  {\bf Algorithm 3} & {\bf Algorithm 1}
 \\ 
\hline

1 &  7 lines  &   7 lines & 10 lines  \\ 
\hline

2 &  7 lines  &   7 lines & 13 lines  \\ 
\hline

3 &  10 lines  &   10 lines & 17 lines \\ 
\hline

4 &  7 lines  &   7 lines & 13 lines   \\ 
\hline

5 &  7 lines  &   7 lines & 10 lines  \\ 
\hline

6 &  7 lines  &   7 lines & 10 lines   \\ 
\hline

7 &  11 lines  &   11 lines & 11 lines   \\ 
\hline

8 &  16 lines  &   20 lines & 18 lines   \\ 
\hline

9 &  9 lines  &   11 lines & 15 lines   \\ 
\hline

10 &  9 lines &   11 lines & 15 lines   \\ 
\hline

11 &  21 lines &   21 lines & 15 lines  \\ 
\hline

12 &  8 lines  &   8 lines & 12 lines  \\ 
\hline

13 &  9 lines  &   9 lines & 17 lines   \\ 
\hline

14 &  9 lines  &   9 lines & 14 lines   \\ 
\hline

15 &  9 lines &   9 lines & 13 lines   \\ 
\hline

16 &  11 lines  &   12 lines & 14 lines    \\ 
\hline

17 &  9 lines  &   9 lines & 12 lines   \\ 
\hline

18 &  7 lines  &   7 lines & 13 lines   \\ 
\hline

19 &  7 lines  &   7 lines & 16 lines  \\ 
\hline

20 &  9 lines  &   11 lines & 11 lines   \\ 
\hline

21 &  10 lines  &   10 lines & 10 lines  \\ 
\hline

22 &  6 lines  &   6 lines & 19 lines  \\ 
\hline
 
23 &  9 lines  &   11 lines & 14 lines   \\ 
\hline

24 &  7 lines  &   7 lines & 15 lines   \\ 
\hline

25 &  9 lines  &   12 lines & 11 lines  \\ 
\hline

\rowcolor{gray!20!}
{\bf Average} &  9.2 lines  &   9.8 lines & 13.5 lines  \\ 
\hline

\end{tabular}

\caption{Ranking of the algorithms}\label{ranking-table}
\end{table}




\begin{center}
\begin{figure}
\begin{tabular}{c c c}

\begin{tikzpicture}[scale=0.3]

\draw (0,12) node [scale=0.5, circle, draw](node1){1};
\draw (12,12) node [scale=0.5, circle, draw](node2){2};
\draw (12,0) node [scale=0.5, circle, draw](node3){3};
\draw (0,0) node [scale=0.5, circle, draw](node4){4};

\draw (6,11) node [scale=0.5, circle, draw](node5){5};
\draw (11,6) node [scale=0.5, circle, draw](node6){6};
\draw (6,1)node [scale=0.5, circle, draw](node7){7};
\draw (1,6)node [scale=0.5, circle, draw](node8){8};

\draw (6,9.5)node [scale=0.5, circle, draw](node9){9};
\draw (9.5,6)node [scale=0.38, circle, draw](node10){\large 10};
\draw (6,2.5)node [scale=0.38, circle, draw](node11){\large 11};
\draw (2.5,6)node [scale=0.38, circle, draw](node12){\large 12};

\draw (4.5,7.5)node [scale=0.38, circle, draw](node13){\large 13};
\draw (7.5,7.5)node [scale=0.38, circle, draw](node14){\large 14};
\draw (7.5,4.5)node [scale=0.38, circle, draw](node15){\large 15};
\draw (4.5,4.5)node [scale=0.38, circle, draw](node16){\large 16};

\draw (node1)--(node2)--(node3)--(node4)--(node1);

\draw (node1)--(node5);
\draw (node1)--(node9);
\draw (node1)--(node8);
\draw (node1)--(node12);

\draw (node2)--(node5);
\draw (node2)--(node9);
\draw (node2)--(node6);
\draw (node2)--(node10);

\draw (node3)--(node6);
\draw (node3)--(node7);
\draw (node3)--(node10);
\draw (node3)--(node11);

\draw (node4)--(node8);
\draw (node4)--(node7);
\draw (node4)--(node12);
\draw (node4)--(node11);

\draw (node8)--(node9)--(node6)--(node11)--(node8);
\draw (node5)--(node10)--(node7)--(node12)--(node5);

\draw (node13)--(node14)--(node15)--(node16)--(node13);

\draw (node5)--(node13);
\draw (node5)--(node14);
\draw (node11)--(node13);
\draw (node11)--(node14);

\draw (node9)--(node15);
\draw (node9)--(node16);
\draw (node7)--(node15);
\draw (node7)--(node16);

\draw (node8)--(node13);
\draw (node8)--(node16);
\draw (node10)--(node13);
\draw (node10)--(node16);

\draw (node6)--(node14);
\draw (node6)--(node15);
\draw (node12)--(node14);
\draw (node12)--(node15);

\node[text width=3cm] at (6.5,-1) 
    {Shrikhande graph};

\end{tikzpicture}

& 

\begin{tikzpicture}[scale=0.3]

\draw (0,0) node [scale=0.5, circle, draw](node4){4};

\draw (11,6) node [scale=0.5, circle, draw](node6){6};

\draw (1,6)node [scale=0.5, circle, draw](node8){8};

\draw (6,2.5)node [scale=0.38, circle, draw](node11){\large 11};
\draw (2.5,6)node [scale=0.38, circle, draw](node12){\large 12};

\draw (4.5,7.5)node [scale=0.38, circle, draw](node13){\large 13};
\draw (7.5,7.5)node [scale=0.38, circle, draw](node14){\large 14};
\draw (7.5,4.5)node [scale=0.38, circle, draw](node15){\large 15};
\draw (4.5,4.5)node [scale=0.38, circle, draw](node16){\large 16};

\draw (node4)--(node8);

\draw (node4)--(node12);
\draw (node4)--(node11);

\draw (node6)--(node11)--(node8);

\draw (node13)--(node14)--(node15)--(node16)--(node13);

\draw (node11)--(node13);
\draw (node11)--(node14);

\draw (node8)--(node13);
\draw (node8)--(node16);

\draw (node6)--(node14);
\draw (node6)--(node15);
\draw (node12)--(node14);
\draw (node12)--(node15);

\node[text width=3cm] at (7.5,-1) 
    {Subgraph $S$};

\end{tikzpicture}

&

\begin{tikzpicture}[scale=0.3]

\draw (0,0) node [scale=0.45, circle, draw](node4){2};

\draw (11,6) node [scale=0.45, circle, draw](node6){6};

\draw (1,6)node [scale=0.45, circle, draw](node8){3};

\draw (6,2.5)node [scale=0.45, circle, draw](node11){1};
\draw (2.5,6)node [scale=0.45, circle, draw](node12){7};

\draw (4.5,7.5)node [scale=0.45, circle, draw](node13){4};
\draw (7.5,7.5)node [scale=0.45, circle, draw](node14){5};
\draw (7.5,4.5)node [scale=0.45, circle, draw](node15){9};
\draw (4.5,4.5)node [scale=0.45, circle, draw](node16){8};

\draw (node4)--(node8);

\draw (node4)--(node12);
\draw (node4)--(node11);

\draw (node6)--(node11)--(node8);

\draw (node13)--(node14)--(node15)--(node16)--(node13);

\draw (node11)--(node13);
\draw (node11)--(node14);

\draw (node8)--(node13);
\draw (node8)--(node16);

\draw (node6)--(node14);
\draw (node6)--(node15);
\draw (node12)--(node14);
\draw (node12)--(node15);

\node[text width=3cm] at (7.5,-1) 
    {Subgraph $S$};

\end{tikzpicture}

\end{tabular}
\caption{The Shrikhande graph and its minimal non-word-representable subgraph $S$ (with the original labelling and the labelling used in our proof)}\label{Shrikhande-pic}
\end{figure}
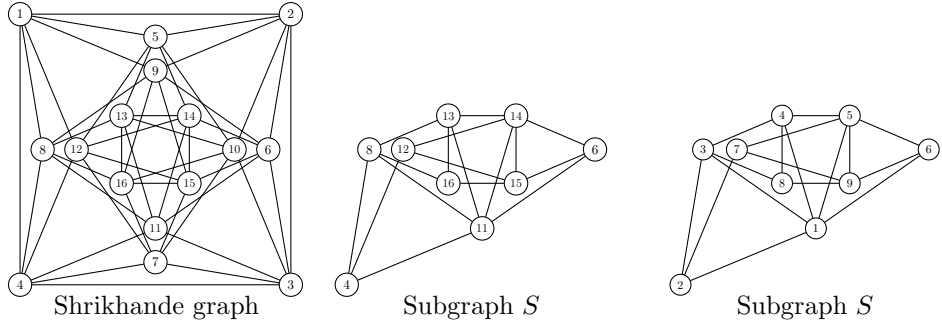
\end{center}

\section{The Shrikhande graph}\label{Shrikhande}
The Shrikhande graph is the graph on 16 vertices and 48 edges in Figure~\ref{Shrikhande-pic}. Among numerous properties of this graph \cite{W-Shrikhande}, it is known for being the smallest {\em distance-regular} graph that is not {\em distance-transitive} \cite[p.\ 136]{BCN89}. In Figure~\ref{Shrikhande-pic}, we also present Shrikhande graph's minimal non-word-representable subgraph $S$ labelled in the original way, and in the way used in the proof below (removing any vertex in $S$ results in a word-representable graph). We next give a proof for non-word-representability of $S$ that will give non-word-representability of the Shrikhande graph taking into account the hereditary nature of word-representability.  Referring to the rightmost graph in Figure~\ref{Shrikhande-pic}, our proof goes as follows. By Theorem~\ref{source-thm} we can assume that vertex 1 is a source and the rest is given by the following 6-lines proof: \\

\begin{small}

\noindent
{\bf 1.}  B5$\rightarrow$6 (Copy 2) O5$\rightarrow$4 (C1-6-5-4) O3$\rightarrow$4 (C1-5-4-3) O3$\rightarrow$2 (C1-4-3-2) B5$\rightarrow$7 (Copy 3) O2$\rightarrow$7 (C1-5-7-2) B3$\rightarrow$8 (Copy 4) O4$\rightarrow$8 (C1-4-8-3) O9$\rightarrow$8 O5$\rightarrow$9 (C4-8-9-5) O6$\rightarrow$9 (C1-6-9-5) O7$\rightarrow$9 (C5-7-9-6) S:3-2-7-9-8

\noindent
{\bf 2.}  MC4 8$\rightarrow$3 O8$\rightarrow$4 (C3-8-4) O9$\rightarrow$7 O8$\rightarrow$9 (C2-7-9-8-3) O5$\rightarrow$9 (C4-8-9-5) O6$\rightarrow$9 (C1-6-9-5) S:5-6-9-7

\noindent
{\bf 3.}  MC3 7$\rightarrow$5 O7$\rightarrow$2 (C1-5-7-2) O7$\rightarrow$9 O9$\rightarrow$6 (C5-7-9-6) O9$\rightarrow$5 (C1-6-9-5) O8$\rightarrow$4 O9$\rightarrow$8 (C4-8-9-5) O8$\rightarrow$3 (C1-4-8-3) S:7-9-8-3-2

\noindent
{\bf 4.}  MC2 6$\rightarrow$5 O4$\rightarrow$5 (C1-6-5-4) O4$\rightarrow$3 (C1-5-4-3) O2$\rightarrow$3 (C1-4-3-2) B5$\rightarrow$7 (Copy~5) O2$\rightarrow$7 (C1-5-7-2) O9$\rightarrow$7 O6$\rightarrow$9 (C5-7-9-6) O5$\rightarrow$9 (C1-6-9-5) O4$\rightarrow$8 O8$\rightarrow$9 (C4-8-9-5) O3$\rightarrow$8 (C1-4-8-3) S:2-3-8-9-7

\noindent
{\bf 5.}  MC5 7$\rightarrow$5 O7$\rightarrow$2 (C1-5-7-2) B3$\rightarrow$8 (Copy 6) O4$\rightarrow$8 (C3-8-4) O7$\rightarrow$9 O9$\rightarrow$8 (C2-7-9-8-3) O9$\rightarrow$5 (C4-8-9-5) O9$\rightarrow$6 (C1-6-9-5) S:7-9-6-5

\noindent
{\bf 6.}  MC6 8$\rightarrow$3 O8$\rightarrow$4 (C1-4-8-3) O8$\rightarrow$9 O9$\rightarrow$5 (C4-8-9-5) O9$\rightarrow$6 (C1-6-9-5) O9$\rightarrow$7 (C5-7-9-6) S:8-9-7-2-3

\end{small}

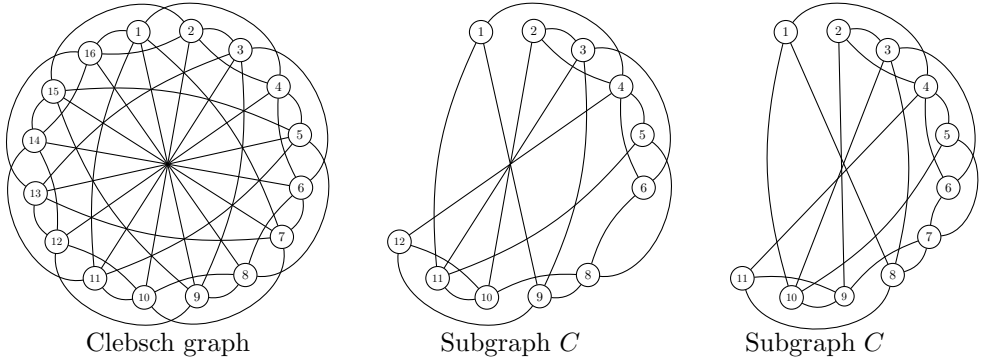
\begin{figure}
\begin{center}
\begin{tabular}{c c c}

\hspace{-1.5cm}
\begin{tikzpicture}[scale=0.3]

    \begin{scope}
    \foreach \a in {1,...,9}
    {
        \node[scale=0.5,circle,draw] (node\a) at ({\a*(-22.5)+125}:6){\a};
    }
    \foreach \a in {10,...,16}
    {
        \node[scale=0.42,circle,draw] (node\a) at ({\a*(-22.5)+125}:6){\a};
    }
        \node[scale=1] (cap) at (-90:8){Clebsch graph};

     \draw [bend left=75] (node1) to (node4);
     \draw [bend left=75] (node3) to (node6);
     \draw [bend left=75] (node5) to (node8);
     \draw [bend left=75] (node7) to (node10);
     \draw [bend left=75] (node9) to (node12);
	 \draw [bend left=75] (node11) to (node14);
	 \draw [bend left=75] (node13) to (node16);
 	 \draw [bend left=75] (node15) to (node2);

     \draw [bend left=30] (node2) to (node3);
     \draw [bend left=30] (node4) to (node5);
     \draw [bend left=30] (node6) to (node7);
 	\draw [bend left=30] (node8) to (node9);
     \draw [bend left=30] (node10) to (node11);
     \draw [bend left=30] (node12) to (node13);
     \draw [bend left=30] (node14) to (node15);
      \draw [bend left=30] (node16) to (node1);

 \draw [bend right=15] (node2) to (node4);
 \draw [bend right=15] (node4) to (node6);
 \draw [bend right=15] (node6) to (node8);
 \draw [bend right=15] (node8) to (node10);
 \draw [bend right=15] (node10) to (node12);
 \draw [bend right=15] (node12) to (node14);
 \draw [bend right=15] (node14) to (node16);
\draw [bend right=15] (node16) to (node2);

\draw (node2)--(node10);
\draw (node4)--(node12);
\draw (node6)--(node14);
\draw (node8)--(node16);
\draw (node1)--(node9);
\draw (node3)--(node11);
\draw (node5)--(node13);
\draw (node7)--(node15);

\draw [bend left=15] (node1) to (node7);
\draw [bend left=15] (node3) to (node9);
\draw [bend left=15] (node5) to (node11);
\draw [bend left=15] (node7) to (node13);
\draw [bend left=15] (node9) to (node15);
\draw [bend left=15] (node11) to (node1);
\draw [bend left=15] (node13) to (node3);
\draw [bend left=15] (node15) to (node5);

    \end{scope}
    
    \end{tikzpicture}
&

\begin{tikzpicture}[scale=0.3]
    \begin{scope}
    \foreach \a in {1,...,6}
    {
        \node[scale=0.5,circle,draw] (node\a) at ({\a*(-22.5)+125}:6){\a};
    }
 \foreach \a in {8,...,9}
    {
        \node[scale=0.5,circle,draw] (node\a) at ({\a*(-22.5)+125}:6){\a};
    }
    \foreach \a in {10,...,12}
    {
        \node[scale=0.42,circle,draw] (node\a) at ({\a*(-22.5)+125}:6){\a};
    }
        \node[scale=1] (cap) at (-90:8){Subgraph $C$};

     \draw [bend left=75] (node1) to (node4);
     \draw [bend left=75] (node3) to (node6);
     \draw [bend left=75] (node5) to (node8);

     \draw [bend left=75] (node9) to (node12);

     \draw [bend left=30] (node2) to (node3);
     \draw [bend left=30] (node4) to (node5);

 	\draw [bend left=30] (node8) to (node9);
     \draw [bend left=30] (node10) to (node11);

 \draw [bend right=15] (node2) to (node4);
 \draw [bend right=15] (node4) to (node6);
 \draw [bend right=15] (node6) to (node8);
 \draw [bend right=15] (node8) to (node10);
 \draw [bend right=15] (node10) to (node12);

\draw (node2)--(node10);
\draw (node4)--(node12);

\draw (node1)--(node9);
\draw (node3)--(node11);

\draw [bend left=15] (node3) to (node9);
\draw [bend left=15] (node5) to (node11);

\draw [bend left=15] (node11) to (node1);

    \end{scope}
    
    \end{tikzpicture}

&

\begin{tikzpicture}[scale=0.3]
    \begin{scope}
    \foreach \a in {1,...,6}
    {
        \node[scale=0.5,circle,draw] (node\a) at ({\a*(-22.5)+125}:6){\a};
    }
 \foreach \a in {7,...,8}
    {
        \node[scale=0.5,circle,draw] (node\a) at ({\a*(-22.5)+125}:6){\a};
    }
    \foreach \a in {9,...,11}
    {
        \node[scale=0.42,circle,draw] (node\a) at ({\a*(-22.5)+125}:6){\a};
    }
        \node[scale=1] (cap) at (-90:8){Subgraph $C$};

     \draw [bend left=75] (node1) to (node4);
     \draw [bend left=75] (node3) to (node6);
     \draw [bend left=75] (node5) to (node7);

     \draw [bend left=75] (node8) to (node11);

     \draw [bend left=30] (node2) to (node3);
     \draw [bend left=30] (node4) to (node5);

 	\draw [bend left=30] (node7) to (node8);
     \draw [bend left=30] (node9) to (node10);

 \draw [bend right=15] (node2) to (node4);
 \draw [bend right=15] (node4) to (node6);
 \draw [bend right=15] (node6) to (node7);
 \draw [bend right=15] (node7) to (node9);
 \draw [bend right=15] (node9) to (node11);

\draw (node2)--(node9);
\draw (node4)--(node11);

\draw (node1)--(node8);
\draw (node3)--(node10);

\draw [bend left=15] (node3) to (node8);
\draw [bend left=15] (node5) to (node10);

\draw [bend left=15] (node10) to (node1);

    \end{scope}
    \end{tikzpicture}

\end{tabular}
\caption{The Clebsch graph and its minimal non-word-representable subgraph~$C$ (with the original labelling and labelling used in the proof)}\label{Clebsch-pic}
\end{center}
\end{figure}

\section{The Clebsch graph}\label{Clebsch}

The Clebsch graph, also known as the {\em Greenwood-Gleason graph} \cite[p. 284]{RW98} and shown in Figure~\ref{Clebsch-pic}, is a strongly regular quintic graph on 16 vertices and 40 edges that enjoys many interesting properties \cite{W-Clebsch}. Figure~\ref{Clebsch-pic} also gives the subgraph $C$ of the Clebsch graph that is confirmed by software \cite{G,S} to be minimal non-word-representable. We next give a 10-line proof of non-word-representability of $C$ that confirms non-word-representability of the Clebsch graph. Referring to the rightmost graph in Figure~\ref{Clebsch-pic}, our proof goes as follows. By Theorem~\ref{source-thm} we can assume that vertex 4 is a source and the rest is given by the following proof: \\

\begin{small}

\noindent
{\bf 1.}  B6$\rightarrow$7 (Copy 2) O5$\rightarrow$7 (C4-6-7-5) B3$\rightarrow$6 (Copy 3) O3$\rightarrow$2 (C2-4-6-3) O3$\rightarrow$8 O8$\rightarrow$7 (C3-8-7-6) B8$\rightarrow$11 (Copy 4) O8$\rightarrow$1 (C1-8-11-4) O10$\rightarrow$1 O3$\rightarrow$10 (C1-10-3-8) O10$\rightarrow$5 (C1-10-5-4) O10$\rightarrow$9 O9$\rightarrow$7 (C5-10-9-7) O2$\rightarrow$9 (C2-9-10-3) O11$\rightarrow$9 (C2-9-11-4) S:8-11-9-7

\noindent
{\bf 2.}  MC4 11$\rightarrow$8 O1$\rightarrow$8 (C1-8-11-4) O9$\rightarrow$7 O11$\rightarrow$9 (C7-9-11-8) O2$\rightarrow$9 (C2-9-11-4) O10$\rightarrow$9 O3$\rightarrow$10 (C2-9-10-3) O1$\rightarrow$10 (C1-10-3-8) O5$\rightarrow$10 (C1-10-5-4) S:5-10-9-7

\noindent
{\bf 3.}  MC3 6$\rightarrow$3 O2$\rightarrow$3 (C2-4-6-3) B2$\rightarrow$9 (Copy 5) O11$\rightarrow$9 (C2-9-11-4) B7$\rightarrow$9 (Copy 6) O5$\rightarrow$10 O10$\rightarrow$9 (C5-10-9-7) O1$\rightarrow$10 (C1-10-5-4) O10$\rightarrow$3 (C2-9-10-3) O8$\rightarrow$3 O1$\rightarrow$8 (C1-10-3-8) O11$\rightarrow$8 (C1-8-11-4) O8$\rightarrow$7 (C3-8-7-6) S:11-8-7-9

\noindent
{\bf 4.}  MC6 9$\rightarrow$7 O11$\rightarrow$8 O8$\rightarrow$7 (C7-9-11-8) O1$\rightarrow$8 (C1-8-11-4) O8$\rightarrow$3 (C3-8-7-6) O1$\rightarrow$10 O10$\rightarrow$3 (C1-10-3-8) O5$\rightarrow$10 (C1-10-5-4) O10$\rightarrow$9 (C2-9-10-3) S:5-10-9-7

\noindent
{\bf 5.}  MC5 9$\rightarrow$2 O9$\rightarrow$10 O10$\rightarrow$3 (C2-9-10-3) O9$\rightarrow$11 (C2-9-11-4) O5$\rightarrow$10 O9$\rightarrow$7 (C5-10-9-7) O1$\rightarrow$10 (C1-10-5-4) O8$\rightarrow$3 O1$\rightarrow$8 (C1-10-3-8) O11$\rightarrow$8 (C1-8-11-4) O8$\rightarrow$7 (C3-8-7-6) S:9-11-8-7

\noindent
{\bf 6.}  MC2 7$\rightarrow$6 O7$\rightarrow$5 (C4-6-7-5) B3$\rightarrow$6 (Copy 7) O3$\rightarrow$2 (C2-4-6-3) B2$\rightarrow$9 (Copy 8) O10$\rightarrow$9 O3$\rightarrow$10 (C2-9-10-3) O11$\rightarrow$9 (C2-9-11-4) O10$\rightarrow$5 O7$\rightarrow$9 (C5-10-9-7) O10$\rightarrow$1 (C1-10-5-4) O3$\rightarrow$8 O8$\rightarrow$1 (C1-10-3-8) O8$\rightarrow$11 (C1-8-11-4) O7$\rightarrow$8 (C3-8-7-6) S:7-8-11-9

\noindent
{\bf 7.}  MC8 9$\rightarrow$2 O9$\rightarrow$11 (C2-9-11-4) B7$\rightarrow$9 (Copy 9) O8$\rightarrow$11 O7$\rightarrow$8 (C7-9-11-8) O8$\rightarrow$1 (C1-8-11-4) O3$\rightarrow$8 (C3-8-7-6) O10$\rightarrow$1 O3$\rightarrow$10 (C1-10-3-8) O10$\rightarrow$5 (C1-10-5-4) O9$\rightarrow$10 (C2-9-10-3) S:7-9-10-5

\noindent
{\bf 8.}  MC9 9$\rightarrow$7 O10$\rightarrow$5 O9$\rightarrow$10 (C5-10-9-7) O10$\rightarrow$1 (C1-10-5-4) O3$\rightarrow$10 (C2-9-10-3) O3$\rightarrow$8 O8$\rightarrow$1 (C1-10-3-8) O8$\rightarrow$11 (C1-8-11-4) O7$\rightarrow$8 (C3-8-7-6) S:9-7-8-11

\noindent
{\bf 9.}  MC7 6$\rightarrow$3 O2$\rightarrow$3 (C2-4-6-3) O8$\rightarrow$3 O7$\rightarrow$8 (C3-8-7-6) B8$\rightarrow$11 (Copy 10) O8$\rightarrow$1 (C1-8-11-4) O7$\rightarrow$9 O9$\rightarrow$11 (C7-9-11-8) O9$\rightarrow$2 (C2-9-11-4) O9$\rightarrow$10 O10$\rightarrow$3 (C2-9-10-3) O10$\rightarrow$1 (C1-10-3-8) O10$\rightarrow$5 (C1-10-5-4) S:7-9-10-5

\noindent
{\bf 10.}  MC10 11$\rightarrow$8 O1$\rightarrow$8 (C1-8-11-4) O1$\rightarrow$10 O10$\rightarrow$3 (C1-10-3-8) O5$\rightarrow$10 (C1-10-5-4) O9$\rightarrow$10 O7$\rightarrow$9 (C5-10-9-7) O9$\rightarrow$2 (C2-9-10-3) O9$\rightarrow$11 (C2-9-11-4) S:7-9-11-8

\end{small}

\section{Concluding remarks}\label{concluding}

In this paper, we introduce methods to generate automatically proofs of non-word-representability of a graph that can be verified, in a robust way, by a human. We do believe that our work and software \cite{S} will have a dramatic impact to the further development of the theory of word-representable graphs. Indeed, now we can argue non-word-representability for many more (larger) graphs without referring to software, which is a very welcoming news. 

As for open problems, we see improving Algorithms 2 and 3 by modifying our approach of selecting edges to branch: for example, we can look for branching edges that increase the number/length of directed paths in the graph, which should increase usability of Lemma~\ref{lemma}. 

Finally, understanding how to estimate, say, the average efficiency of our algorithms, or relevant algorithms yet to be introduced, in terms of certain parameters (number of cycles or alike) is a good theoretical question that seems to be very challenging. The time complexity of our algorithms is also a very interesting and challenging direction of research that was completely ignored by us because our focus was in producing short proofs. 

\section*{Acknowledgments}

The first author is grateful to the London Mathematical Society for supporting work on this project under ``Computer Science Small Grants -- Scheme 7''.


\begin{thebibliography}{20}

\bibitem{BCN89} A. E. Brouwer, A. M. Cohen, and A. Neumaier. Distance-Regular Graphs. New York: Springer-Verlag, 1989.

\bibitem{G} M. Glen. Software available at  \verb>personal.strath.ac.uk/> \verb>sergey.kitaev/word-representable-graphs.html>

\bibitem{G76} M.\ C.\ Golumbic. The complexity of Comparability Graph Recognition and Coloring, {\em Computing} {\bf 18} (1977) 199--208.

\bibitem{HKP16} M. M. Halld\'{o}rsson, S. Kitaev, A. Pyatkin. Semi-transitive orientations and word-representable graphs, {\em Discr. Appl. Math.} {\bf 201} (2016) 164--171.

\bibitem{K17} S. Kitaev. A Comprehensive Introduction to the Theory of Word-Representable Graphs. {\em Lect. Notes in Comp. Sci.} {\bf 10396} (2017) 36--67. 

\bibitem{KL15} S. Kitaev and V. Lozin. Words and Graphs, {\em Springer}, 2015.

\bibitem{KP08} S. Kitaev, A. Pyatkin. On representable graphs. {\em
J. Autom., Lang. and Combin.} {\bf 13} (2008) 1, 45--54.

\bibitem{KP} S. Kitaev and A. Pyatkin. On semi-transitive orientability of triangle-free graphs. {\em Discussiones Mathematicae Graph Theory} {\bf 43} (2023), ID: 4621, page 533.

\bibitem{RW98} R. C. Read and R. J. Wilson. {\em An Atlas of Graphs.} Oxford, England: Oxford University Press, 1998. 

\bibitem{S} H. Sun. Software available at  \\ \verb>personal.strath.ac.uk/sergey.kitaev/human-verifiable-proofs.html>

\bibitem{W-Clebsch} E. W. Weisstein. ``Clebsch Graph.'' From MathWorld--A Wolfram Web Resource. \verb>https://mathworld.wolfram.com/ClebschGraph.html>

\bibitem{W-Shrikhande} E. W. Weisstein. ``Shrikhande Graph.'' From MathWorld--A Wolfram Web Resource. \verb>https://mathworld.wolfram.com/ShrikhandeGraph.html>

\bibitem{Z} H. Zantema. Software REPRNR to compute the representation number of a graph, 2018. Available at \verb>https://www.win.tue.nl/~hzantema/reprnr.html.>

\end{thebibliography}
\end{document}